\documentclass{tglat2e-arxiv}
\usepackage[colorlinks=true, pdfstartview=FitV, linkcolor=blue, citecolor=blue, urlcolor=blue, breaklinks=true]{hyperref}
\usepackage{amsmath,amsfonts,amssymb,amscd,comment,euscript}
\usepackage[usenames]{color}
\usepackage[all]{xy}
\usepackage{tikz}

\newcommand\C{\mathbb{C}}
\newcommand\Z{\mathbb{Z}}
\newcommand\Q{\mathbb{Q}}

\newcommand\N{\mathbb{N}}
\newcommand\kk{{\Bbbk}}

\newcommand\ev{\mathrm{ev}}

\newcommand\Vir{\mathrm{Vir}}

\newcommand\g{\mathfrak{g}}

\newcommand\fm{\mathfrak{m}}

\newcommand\fa{\mathfrak{a}}

\newcommand\cV{\mathcal{V}}
\newcommand\cB{\mathcal{B}}

\DeclareMathOperator{\End}{End}

\DeclareMathOperator{\Spec}{Spec}
\DeclareMathOperator{\maxSpec}{maxSpec}
\DeclareMathOperator{\Der}{Der}

\DeclareMathOperator{\Ann}{Ann}
\DeclareMathOperator{\Supp}{Supp} 

\DeclareMathOperator{\height}{ht}

\DeclareMathOperator{\rad}{rad}

\DeclareMathOperator{\hm}{hm}

\theoremstyle{plain}
\newtheorem{theorem}{Theorem}[section]
\newtheorem*{theorem*}{Theorem}
\newtheorem{proposition}[theorem]{Proposition}
\newtheorem{lemma}[theorem]{Lemma}
\newtheorem{corollary}[theorem]{Corollary}

\theoremstyle{definition}
\newtheorem{definition}[theorem]{Definition}
\newtheorem*{remark*}{Remark}
\newtheorem{remark}[theorem]{Remark}

\numberwithin{equation}{section}

\renewcommand{\theenumi}{\alph{enumi}}

\allowdisplaybreaks

\begin{document}
%

\title[Quasifinite modules over map Virasoro algebras]{Classification of irreducible\\ quasifinite modules over\\ map Virasoro algebras}

\authors{Alistair Savage\thanks{This work was supported by a Discovery Grant from the Natural Sciences and Engineering Research Council of Canada.}
\address{Department of Mathematics \\ and Statistics\\ University of Ottawa, Canada}
\email{alistair.savage@uottawa.ca}
}



\received{August 2, 2011}
\accepted{January 14, 2012}

\maketitle 

\begin{abstract}
We give a complete classification of the irreducible quasifinite modules for algebras of the form $\Vir \otimes A$, where $\Vir$ is the Virasoro algebra and $A$ is a finitely generated commutative associative unital algebra over the complex numbers.  It is shown that all such modules are tensor products of generalized evaluation modules.  We also give an explicit sufficient condition for a Verma module of $\Vir \otimes A$ to be reducible.  In the case that $A$ is an infinite-dimensional integral domain, this condition is also necessary.
\end{abstract}

\tableofcontents

%
\section*{Introduction}
%

The \emph{Witt algebra} $\Der \C[t,t^{-1}]$ has basis $d_n := t^{n+1} \frac{d}{dt}$, $n \in \Z$, and Lie bracket given by $[d_m,d_n] = (n-m)d_{n+m}$.  It is the Lie algebra of polynomial vector fields on $S^1$ (or $\C^*$).  The \emph{Virasoro algebra} $\Vir := \Der \C[t,t^{-1}] \oplus \C c$ is the universal central extension of the Witt algebra.  It has Lie bracket
\[
  [d_n,c]=0,\quad [d_m,d_n] = (n-m)d_{m+n} + \delta_{m,-n} \frac{m^3-m}{12}c,\quad m,n \in \Z.
\]
The Virasoro algebra plays a fundamental role in the theory of vertex operator algebras, conformal field theory, string theory, and the representation theory of affine Lie algebras.

An important class of modules for the Virasoro algebra are the so-called \emph{quasifinite modules} (or \emph{Harish-Chandra modules}), which are modules on which the maximal abelian diagonalizable subalgebra $\C d_0 \oplus \C c$ acts reductively with finite-dimen\-sional weight spaces.  The unitary irreducible quasifinite $\Vir$-modules were classified by Chari and Pressley \cite{CP88}.  All irreducible quasifinite $\Vir$-modules (without the assumption of unitarity) were then classified by Mathieu \cite{Mat92}, where it was shown that they are all highest weight modules, lowest weight modules or \emph{modules of the intermediate series} (otherwise known as \emph{tensor density modules} and whose nonzero weight spaces are all one-dimensional).

Many generalizations of the Virasoro algebra and other closely related algebras have been considered by several authors.  These include, but are not limited to, the higher rank Virasoro algebras \cite{LZ06,Maz99,Su01,Su03}, the $\Q$-Virasoro algebra \cite{Maz00}, the generalized Virasoro algebras \cite{BZ04,GLZ06,HWZ03}, the twisted Heisenberg-Virasoro algebra \cite{LZ10}, and the loop-Virasoro algebra \cite{GLZ08}.  In many cases, classifications of the irreducible quasifinite modules have been given.

The goal of the current paper is to classify the irreducible quasifinite modules for \emph{map Virasoro algebras}, which are Lie algebras of the form $\Vir \otimes A$, where $A$ is a finitely generated commutative associative unital algebra.  The related problem of classifying the irreducible finite-dimensional modules for $\g \otimes A$, where $\g$ is a finite-dimensional Lie algebra, as well as for the fixed point algebras of $\g \otimes A$ under certain finite group actions (the \emph{equivariant map algebras}), was solved in \cite{CFK10,NSS09}.  In particular, all irreducible finite-dimensional modules are tensor products of one-dimensional modules and \emph{evaluation modules}.  The main result (Theorem~\ref{thm:main-theorem}) of the current paper is the following (we refer the reader to Section~\ref{sec:map-algebras} for the definitions of evaluation and generalized evaluation modules).

\begin{theorem*}
  Any irreducible quasifinite $(\Vir \otimes A)$-module is one of the following:
  \begin{enumerate}
    \item \label{thm-item:int-series} a single point evaluation module corresponding to a $\Vir$-module of the intermediate series,
    \item a finite tensor product of single point generalized evaluation modules corresponding to irreducible highest weight modules, or
    \item a finite tensor product of single point generalized evaluation modules corresponding to irreducible lowest weight modules.
  \end{enumerate}
  In particular, they are all tensor products of single point generalized evaluation modules.
\end{theorem*}
We note that the problem of determining which highest and lowest weight irreducible modules are quasifinite is nontrivial when $A$ is infinite-dimensional.  (When $A$ is finite-dimensional, for instance when $A = \C$ and $\Vir \otimes A$ is just the usual Virasoro algebra, all highest and lowest weight irreducible modules are quasifinite.)

We also give an explicit sufficient condition for the Verma modules of $\Vir \otimes A$ to be reducible.  Under the additional assumption that $A$ is an infinite-dimensional integral domain, the condition is also necessary (Theorem~\ref{thm:Verma-reducible}).

Owing to the fact that the Virasoro algebra is infinite-dimensional, the techniques used in the current paper are very different than those used in \cite{NSS09}.  We also see some differences in the classifications.  In particular, we see that the modules of type~\eqref{thm-item:int-series} in the above theorem can only have support at a single point.  This is due to the fact that a tensor product of such modules no longer has finite-dimensional weight spaces.

The Lie algebra $\Vir \otimes A$ can be thought of as a central extension of the Lie algebra of the group of diffeomorphisms of $(\Spec A) \times \C^*$ fixing the first factor.  For this reason, we hope the results of the current paper will be useful in addressing the important open problem of classifying the quasifinite modules for the Lie algebra of polynomial vector fields on more arbitrary varieties (see, for example, \cite{Rao04} for a conjecture related to the case of the higher dimensional torus).  When $A = \C[t,t^{-1}]$, the Lie algebra $\Vir \otimes A = \Vir \otimes \C[t,t^{-1}]$ is called the \emph{loop-Virasoro algebra}.  In this case, the results of the current paper recover those of \cite{GLZ08}.  In fact, many of our arguments are inspired by ones found there.

There remain many interesting open questions related to the representation theory of the Virasoro algebra and its generalizations.  For the map Virasoro algebras, it would be useful to describe the extensions between irreducible quasifinite modules.  This was done for the usual Virasoro algebra in \cite{MP91a,MP91b,MP92} and for the equivariant map algebras in \cite{NS11}.  It would also be interesting to see if a classification of the irreducible quasifinite modules for twisted (or equivariant) versions of map Virasoro algebras is possible.  Finally, one might hope for a classification similar to the one in the current paper (in terms of generalized evaluation modules) when $\Vir$ is replaced by other important infinite-dimensional Lie algebras such as the Heisenberg algebra or the Lie algebra of all differential operators on the circle (instead of just those of order one).

The paper is organized as follows.  In Section~\ref{sec:map-algebras} we review some important definitions and results for map algebras (Lie algebras of the form $\g \otimes A$).  We introduce the Virasoro algebra and its generalization considered in the current paper in Section~\ref{sec:Virasoro}.  In Section~\ref{sec:wt-spaces} we show that any quasifinite module is either a highest weight module, a lowest weight module, or a module whose weight space dimensions are uniformly bounded.  We then classify the uniformly bounded modules in Section~\ref{sec:UBM} and the highest/lowest weight modules in Section~\ref{sec:hw-modules}.  Finally, in Section~\ref{sec:Verma-reducibility} we describe a necessary and sufficient condition for the Verma modules to be reducible.

\subsection*{Notation}

Throughout, $A$ will denote a finitely generated (hence Noetherian) commutative associative unital algebra over the field $\C$ of complex numbers; and all tensor products, Lie algebras, vector spaces, etc., are over $\C$.  When we refer to the \emph{dimension} of $A$, we are speaking of its dimension as a complex vector space (as opposed to referring to a geometric dimension).  Similarly, when we say that an ideal $J \trianglelefteq A$ has \emph{finite codimension} in $A$, we mean that the dimension of $A/J$ as a complex vector space is finite.  We let $\N$ be the set of nonnegative integers and $\N_+$ be the set of positive integers.  For a Lie algebra $L$, $U(L)$ will denote its universal enveloping algebra.  This has a natural filtration $U_0(L) \subseteq U_1(L) \subseteq U_2(L) \subseteq \dots$ coming from the grading on the tensor algebra of $L$.

\subsection*{Acknowledgements}

The author would like to thank the Institut de Math\'ematiques de Jussieu and the D\'epartement de Math\'ematiques d'Orsay for their hospitality during his stays there, when the writing of the current paper took place.  He would also like to thank Y.~Billig, E.~Neher, O.~Schiffmann, G.~Smith, and K.~Zhao for useful discussions and D.~Daigle for providing proofs of various commutative algebra results used in Section~\ref{sec:Verma-reducibility}.

%
\section{Map algebras} \label{sec:map-algebras}
%

In this section we review some important definitions and results related to map algebras.

\begin{definition}[Map algebra]
  If $\g$ is a Lie algebra, then $\g \otimes A$ is the \emph{map algebra} associated to $\g$ and $A$.  It is a Lie algebra with bracket defined by
  \[
    [u_1 \otimes f_1, u_2 \otimes f_2] = [u_1,u_2] \otimes f_1f_2
  \]
  (extended by linearity).  We will identify $\g$ with the Lie subalgebra $\g \otimes \C \subseteq \g \otimes A$.
\end{definition}

Recall that a Lie algebra $\g$ is said to be \emph{perfect} if $[\g,\g]=\g$.

\begin{lemma}
  Suppose $\g$ is a perfect Lie algebra and $V$ is a $(\g \otimes A)$-module.  Then
  \[
    \{f \in A \ |\ (\g \otimes f) V = 0\}
  \]
  is an ideal of $A$.
\end{lemma}

\begin{proof}
  Let $J = \{f \in A \ |\ (\g \otimes f) V = 0\}$.  Clearly $J$ is a linear subspace of $A$.  Suppose $f \in J$ and $g \in A$.  Since $\g$ is perfect, for all $u \in \g$, we have $u = \sum_{i=1}^n [u_i,u'_i]$ for some $u_i,u'_i \in \g$, $i=1,\dots,n$.  Then
  \[
    (u \otimes fg) V = \left(\sum_{i=1}^n [u_i \otimes f, u'_i \otimes g]\right) V = 0.
  \]
  Hence $J$ is an ideal of $A$.\qed
\end{proof}

For the rest of the paper, we assume that $\g$ is perfect.  (Later we shall take $\g$ to be the Virasoro algebra, which is perfect.)

\begin{definition}[Support]
  For a $(\g \otimes A)$-module $V$, we define
  \begin{align*}
    \Ann_A V &:= \{f \in A \ |\ (\g \otimes f) V = 0\} \trianglelefteq A, \\
    \Supp_A V &:= \{\fm \in \maxSpec A\ |\ \Ann_A V \subseteq \fm\}.
  \end{align*}
  The set $\Supp_A V$ is called the \emph{support} of $V$.  We say $V$ has \emph{finite support} if $\Supp_A V$ is finite.
\end{definition}

\begin{definition}[Evaluation module]
  Suppose $\fm \trianglelefteq A$ is a maximal ideal and $V$ is a $\g$-module with corresponding representation $\rho : \g \to \End V$.  Then the composition
  \[
    \g \otimes A \twoheadrightarrow (\g \otimes A)/(\g \otimes \fm) \cong \g \otimes (A/\fm) \cong \g \xrightarrow{\rho} \End V,
  \]
  is called a \emph{(single point) evaluation representation} of $\g \otimes A$.  The corresponding module is called a \emph{(single point) evaluation module} and is denoted $\ev_\fm V$.
\end{definition}

\begin{definition}[Generalized evaluation module]
  Suppose $\fm \trianglelefteq A$ is a maximal ideal, $n \in \N_+$, and $V$ is a $(\g \otimes (A/\fm^n))$-module with corresponding representation $\rho : \g \otimes (A/\fm^n) \to \End V$.  Then the composition
  \[
    \g \otimes A \twoheadrightarrow (\g \otimes A)/(\g \otimes \fm^n) \cong \g \otimes (A/\fm^n) \xrightarrow{\rho} \End V
  \]
  is called a \emph{(single point) generalized evaluation representation} of $\g \otimes A$.  The corresponding module is called a \emph{(single point) generalized evaluation module} and is denoted $\ev_{\fm^n} V$.
\end{definition}

%
\section{Map Virasoro algebras} \label{sec:Virasoro}
%

In this section we define the Virasoro algebra and its generalizations, the map Virasoro algebras.  We also review the classification of irreducible quasifinite modules for the Virasoro algebra.

\begin{definition}[Virasoro algebra $\Vir$ and map Virasoro algebra $\cV$]
  The \linebreak \emph{Virasoro algebra} $\Vir$ is the Lie algebra with basis $\{c, d_n\ |\ n \in \Z\}$ and Lie bracket given by
  \[
    [d_n,c]=0,\quad [d_m,d_n] = (n-m)d_{m+n} + \delta_{m,-n} \frac{m^3-m}{12}c,\quad m,n \in \Z.
  \]
  We define $\cV = \Vir \otimes A$ and call this a \emph{map Virasoro algebra}.
\end{definition}

We have a decomposition
\[
  \cV  = \bigoplus_{i \in \Z} \cV_i,\quad \cV_0 = (d_0 \otimes A) \oplus (c \otimes A),\quad \cV_i = d_i \otimes A,\ i \ne 0,
\]
which is simply the weight decomposition of $\Vir$, that is, the eigenspace decomposition corresponding to the action of $d_0$.  Set
\[
  \cV_+ = \bigoplus_{i > 0} \cV_i,\quad \cV_- = \bigoplus_{i < 0} \cV_i, \quad \cV_{\ge n} = \bigoplus_{i \ge n} \cV_i,\ n \in \Z.
\]

For a $\cV$-module $V$ and $\lambda \in \C$, we let $V_\lambda$ be the eigenspace (or \emph{weight space}) corresponding to the action of $d_0$ with eigenvalue $\lambda$.  We say $V$ is a \emph{weight module} if $V = \bigoplus_{\lambda \in \C} V_\lambda$.  We shall use the following lemma repeatedly without mention.

\begin{lemma} \label{lem:weight-string}
  Any irreducible weight $\cV$-module $V$ has a weight decomposition of the form $V = \bigoplus_{i \in \Z} V_{\alpha+i}$ for some $\alpha \in \C$.
\end{lemma}

\begin{proof}
  This follows immediately from the fact that any nonzero weight vector generates $V$.\qed
\end{proof}

\begin{definition}[Quasifinite module]
  A $\cV$-module is called a \emph{quasifinite module} (or a \emph{Harish-Chandra module}, or an \emph{admissible module}) if it is a weight module and all weight spaces are finite-dimensional.
\end{definition}

\begin{definition}[Highest and lowest weight modules]
  A $\cV$-module $V$ is called a \emph{highest weight module} (respectively, \emph{lowest weight module}) if there exists a nonzero weight vector (that is, eigenvector of $d_0$) $v \in V$ with $\cV_+ v = 0$ (respectively, $\cV_- v =0$) and $U(\cV)v = V$.  Such a vector $v$ is called a \emph{highest weight vector} (respectively, \emph{lowest weight vector}).
\end{definition}

\begin{remark}
  An irreducible $\cV$-module $V$ is a highest (resp.\ lowest) weight module if and only if it is a weight module and there exists a nonzero vector $v \in V$ with $\cV_+ v=0$ (resp.\ $\cV_- v=0$).  Indeed, writing $v$ as a sum of (nonzero) weight vectors, we see that its term of highest weight must also be annihilated by $\cV_+$ (resp.\ $\cV_-$) and generates $V$ (since $V$ is irreducible).
\end{remark}

\begin{remark}
  Via the involution of $\Vir$ (hence of $\cV$) given by $d_n \mapsto -d_{-n}$, $n \in \Z$, $c \mapsto -c$, one can translate between highest weight and lowest weight modules.  Thus, we will often prove results only for highest weight modules, with the corresponding results for lowest weight modules following from this translation.
\end{remark}

By the PBW Theorem, we have a triangular decomposition
\[
  U(\cV) \cong U(\cV_-) \otimes U(\cV_0) \otimes U(\cV_+).
\]
Note that since $\cV_0$ is abelian, any one-dimensional representation of $\cV_0$ (equivalently, of $U(\cV_0)$) is simply a linear map from $\cV_0$ to the ground field $\C$.  For such a linear map $\varphi$, let $\C_\varphi$ denote the corresponding module.

\begin{definition}[Verma module]
  Let $\varphi \in \hom_\C(\cV_0,\C)$ be a one-dimensional representation of $\cV_0$.  Extend $\C_\varphi$ to a module for $\cV_0 \oplus \cV_+$ by defining $\cV_+$ to act by zero.  Then
  \[
    M(\varphi) := U(\cV) \otimes_{U(\cV_0 \oplus \cV_+)} \C_\varphi
  \]
  is the \emph{Verma module} corresponding to $\varphi$.  It is a highest weight module of highest weight $\varphi(d_0)$ and $M(\varphi) = \bigoplus_{i \in \N} M(\varphi)_{\varphi(d_0)-i}$.  We define $\tilde v_\varphi := 1 \otimes 1_\varphi$, where $1_\varphi$ denotes the unit in $\C_\varphi$.  Thus $\tilde v_\varphi$ is a highest weight vector of $M(\varphi)$.  Note that $M(\varphi) \cong U(\cV_-)$ as $U(\cV_-)$-modules.
\end{definition}

\begin{definition}[Irreducible highest weight module]
  For $\varphi \in \hom_\C(\cV_0,\C)$, let $N(\varphi)$ be the unique maximal proper submodule of $M(\varphi)$.  Then
  \[
    V(\varphi) := M(\varphi)/N(\varphi)
  \]
  is the \emph{irreducible highest weight module} corresponding to $\varphi$.  It is a highest weight module of highest weight $\varphi(d_0)$ and $V(\varphi) = \bigoplus_{i \in \N} V(\varphi)_{\varphi(d_0)-i}$.  We denote the image of $\tilde v_\varphi$ in $V(\varphi)$ by $v_\varphi$.  In the case that $A \cong \C$, so $\cV \cong \Vir$, $\varphi$ is uniquely determined by $\varphi(c)$ and $\varphi(d_0)$.  We will therefore sometimes write $V(\varphi(c),\varphi(d_0))$ for $V(\varphi)$.
\end{definition}

\begin{definition}[Uniformly bounded module]
  A weight $\cV$-module $V$ is called \emph{uniformly bounded} if there exists $N \in \N$ such that $\dim V_\lambda < N$ for all $\lambda \in \C$.
\end{definition}

Note that the above definitions apply to the usual Virasoro algebra since $\Vir \cong \cV$ when $A = \C$.  In this case, they reduce to the definitions appearing in the literature.  We now summarize some known results on quasifinite modules for $\Vir$.

Note that $\Der \C[t,t^{-1}]$ acts naturally on $\C[t,t^{-1}]$ and therefore so does $\Vir$, with $c$ acting as zero.  Twistings of this action yield the following important $\Vir$-modules.

\begin{definition}[Module of the intermediate series] \label{def:intermediate-series}
  Fix $a,b \in \C$.  Define $V(a,b)$ to be the $\Vir$-module with underlying vector space $\C[t,t^{-1}]$, with $c$ acting by zero, and
  \[
    u \cdot v = (u + a\, \mathrm{div} (u) + bt^{-1} ut)v,\quad \forall\ u \in \Der \C[t,t^{-1}],\ v \in V(a,b),
  \]
  where $\mathrm{div} \left( p(t) \frac{d}{dt} \right) = \frac{d}{dt} p(t)$ for a polynomial $p(t) \in \C[t]$.  If $b \not \in \Z$ or $a \ne 0,1$, then $V(a,b)$ is irreducible (see, for example, \cite[Proposition~1.1]{KR87}).  Otherwise, $V(a,b)$ has two irreducible subquotients: the trivial submodule $\C$ and $V(a,b)/\C$.  The nontrivial irreducible subquotients of the modules $V(a,b)$ are called \emph{modules of the intermediate series} (or \emph{tensor density modules}).
\end{definition}

We record the following result since it will be used several times in the current paper.

\begin{lemma} \label{lem:inter-series-dim-1-wt-spaces}
  If $V$ is a module of the intermediate series for $\Vir$, then $V$ is a weight module.  Furthermore, if we write $V = \bigoplus_{i \in \Z} V_{\alpha + i}$ for some $\alpha \in \C$ as in Lemma~\ref{lem:weight-string}, then $\dim V_{\alpha+i} = 1$ for all $i \in \Z$ with $\alpha + i \ne 0$.
\end{lemma}

\begin{proof}
  This follows immediately from Definition~\ref{def:intermediate-series}.\qed
\end{proof}

The following result gives a classification of the irreducible quasifinite modules for $\Vir$ (see also \cite[Theorem~0.5]{CP88} for an earlier classification under the additional assumption of unitarity).

\begin{proposition}[{\cite[Theorem~1]{Mat92}}] \label{prop:Vir-module-classification}
  Any irreducible quasifinite module over $\Vir$ is a highest weight module, a lowest weight module, or a module of the intermediate series.
\end{proposition}

\begin{corollary} \label{cor:Vir-UBM=int-series}
  Any nontrivial uniformly bounded irreducible quasifinite $\Vir$-module is a module of the intermediate series.
\end{corollary}

\begin{proof}
  It is shown in \cite[Corollary~III.3]{MP91a} that the nontrivial highest and lowest weight $\Vir$-modules are not uniformly bounded.  The result then follows from Proposition~\ref{prop:Vir-module-classification}.\qed
\end{proof}

%
\section{Dimensions of weight spaces} \label{sec:wt-spaces}
%

In this section, we prove an important result about the behavior of dimensions of weight spaces of quasifinite $\cV$-modules.  This is an analogue of \cite[Lemma~1.7]{Mat92} for the classical Virasoro algebra $\Vir$.  It was proven in \cite[Theorem~3.1]{GLZ08} for the case $A = \C[t,t^{-1}]$.

\begin{proposition} \label{prop:wt-spaces}
  Every irreducible quasifinite $\cV$-module is a highest weight module, a lowest weight module, or a uniformly bounded module.
\end{proposition}

\begin{proof}
  Let $V$ be an irreducible quasifinite $\cV$-module that is not uniformly bounded.  Let $W$ be a minimal $\Vir$-submodule of $V$ such that $V/W$ is trivial as a $\Vir$-module, and let $T$ be the maximal trivial $\Vir$-submodule of $W$.   Set $\bar W = W/T$.  By \cite[Theorem~3.4]{MP91b}, there exists a $\Vir$-module decomposition $\bar W = \bar W^+ \oplus \bar W^- \oplus \bar W^0$, where the weights of $\bar W^+$ are bounded above, those of $\bar W^-$ are bounded below, and $\bar W^0$ is uniformly bounded.  Without loss of generality, we assume $\bar W^+$ is nonzero.  For any $w \in W$, we denote its image in $\bar W$ by $\bar w$.  To show that $V$ is a highest weight module it suffices, by \cite[Lemma~1.6]{Mat92}, to show that there exists a nonzero $v \in V$ such that $\cV_{\ge n} v = 0$ for some $n \in \Z$.

  Since $V$ is irreducible, the central element $c$ acts as a constant $c'$ by Schur's Lemma.  Note that if $c' \ne 0$, then $V$ can have no trivial subquotients (in particular, $W=V$ and $T=0$).  Suppose $c'=0$ and the maximum weight of $\bar W^+$ is zero.  If we let $w \in W$ such that $\bar w$ is a nonzero vector of weight zero, then $U(\Vir)w/(U(\Vir)w \cap T) \subseteq \bar W$ is a highest weight module of highest weight zero which is nontrivial by our definition of $\bar W$.  Since its irreducible quotient is the trivial module, it must contain highest weight vectors of nonzero highest weight.  Choose $v \in W$ so that $\bar v$ is such a vector and let $\lambda$ be its weight.  In the other cases (i.e., $c' \ne 0$ or the maximum weight of $\bar W^+$ is nonzero), let $\lambda$ be the maximum weight of $\bar W^+$ and let $v \in W$ such that $\bar v$ is a nonzero highest weight vector of weight $\lambda$.

  Let $M = U(\Vir) v$.  Then $M/(M \cap T) \subseteq \bar W^+$ is a nontrivial highest weight $\Vir$-module of highest weight $\lambda$.  Let $M'$ be the largest $\Vir$-submodule of $M$ with $M'_\lambda=0$.  Then $M \cap T \subseteq M'$ and $M/M'$ is isomorphic to the nontrivial irreducible $\Vir$-module $V(c',\lambda)$.  It follows from \cite[Corollary~III.3]{MP91a} that $V(c',\lambda)$ is not uniformly bounded.  Thus there exists $k \in \N$ such that $\dim V(c',\lambda)_{\lambda-k} > \dim V_\lambda$.  For $f \in A$, consider the linear map
  \[
    d_k \otimes f : M''_{\lambda-k} \to V_\lambda,
  \]
  where $M''_{\lambda-k}$ is a vector space complement to $M'_{\lambda-k}$ in $M_{\lambda-k}$.  Since
  \[
    \dim M''_{\lambda-k} = \dim V(c',\lambda)_{\lambda-k} > \dim V_\lambda,
  \]
  this map has nonzero kernel.  Thus there exists a nonzero $w_f \in M''_{\lambda-k}$ such that $(d_k \otimes f)w_f = 0$.

  Let $N=\max(1,-\lambda,-2k)$.  Then, for all $j > N$, we have $d_{k+j} w_f \in M_{\lambda + j} = 0$.  Thus
  \begin{equation} \label{eq:wf-killed}
    (d_{2k+j} \otimes f)w_f = -\frac{1}{j}[d_{k+j},d_k \otimes f]w_f = 0 \quad \forall\ j > N.
  \end{equation}
  Since $w_f \in M_{\lambda-k} \setminus M'_{\lambda-k}$, there exists $z \in \C$ and $i_1,\dots,i_r \in \N_+$ with $i_1 + \dots + i_r = k$ such that $z d_{i_1} \cdots d_{i_r} w_f = v$.  Using~\eqref{eq:wf-killed}, for $j > N$ we have
  \begin{align*}
    (d_{2k+j} \otimes f)v &= (d_{2k+j} \otimes f)(zd_{i_1} \cdots d_{i_r})w_f \\
    &=  [d_{2k+j} \otimes f,zd_{i_1} \cdots d_{i_r}] w_f \\
    &= z \sum_{\ell=1}^r  (i_\ell - 2k - j) d_{i_1} \cdots d_{i_{\ell-1}} (d_{2k+j+i_\ell} \otimes f) d_{i_\ell+1} \cdots d_{i_r} w_f.
  \end{align*}
  Continuing to move the terms of the form $d_m \otimes f$ to the right and using~\eqref{eq:wf-killed}, we see that
  \[
    (d_{2k+j} \otimes f)v = 0 \quad \forall\ f \in A,\ j > N.
  \]
  In other words $\cV_{\ge 2k+N+1} v = 0$, completing the proof.\qed
\end{proof}

%
\section{Uniformly bounded modules} \label{sec:UBM}
%

In this section, we classify the uniformly bounded $\cV$-modules.  We show that they are all single point evaluation modules corresponding to $\Vir$-modules of the intermediate series.  In the case $A = \C[t,t^{-1}]$, this was proven in \cite[Theorem~5.1]{GLZ08}.

\begin{proposition} \label{prop:irred-killed-by-finite-codim-ideal}
  Suppose $V$ is a uniformly bounded irreducible $\cV$-module.  Then $(\Vir \otimes J) V = 0$ for some ideal $J \trianglelefteq A$ of finite codimension.  In particular, the uniformly bounded irreducible $\cV$-modules have finite support.
\end{proposition}

\begin{proof}
  If $V$ is trivial, we simply take $J=A$.  We therefore assume that $V$ is nontrivial.  We have a weight space decomposition $V = \bigoplus_{i \in \Z} V_{\alpha + i}$ for some $\alpha \in \C$.  Since $V$ is uniformly bounded, we can choose $N \in \N$ such that $\dim V_{\alpha+i} \le N$ for all $i \in \Z$.  Fix $i \in \Z$ such $V_{\alpha+i} \ne 0$.  For $j \in \Z \setminus \{0\}$, define
  \[
    I_j = \{f \in A\ |\ (d_j \otimes f) V_{\alpha+i} = 0\}.
  \]
  Clearly, $I_j$ is a linear subspace of $A$.  For any $f \in I_j$, $g \in A$, and $v \in V_{\alpha+i}$, we have
  \[
    j(d_j \otimes gf) v = [d_0 \otimes g, d_j \otimes f] v = (d_0 \otimes g)(d_j \otimes f)v - (d_j \otimes f)(d_0 \otimes g)v = 0,
  \]
  where we have used the fact that elements of $d_0 \otimes A$ preserve weights.  Thus $I_j$ is an ideal of $A$ for all $j \in \Z \setminus \{0\}$.  Since $I_j$ is the kernel of the linear map
  \[
    A \to \hom_\C(V_{\alpha+i}, V_{\alpha+i+j}),\quad f \mapsto (v \mapsto (d_j \otimes f)v),
  \]
  we have that $\dim A/I_j \le \dim \hom_\C(V_{\alpha+i},V_{\alpha+i+j}) \le N^2$ for all $j \in \Z \setminus \{0\}$.

  We claim that $I_1^j I_2 \subseteq I_{j+2}$ for all $j \ge 1$.  Since
  \[
    (d_3 \otimes f_1 f_2)V_{\alpha+i} = [d_1 \otimes f_1, d_2 \otimes f_2] V_{\alpha+i} = 0 \quad \forall\ f_1 \in I_1,\ f_2 \in I_2,
  \]
  the case $j=1$ is proved.  Assume the result is true for some fixed $j \ge 1$.  Then
  \[
    (j+1)(d_{j+3} \otimes f_1 f)V_{\alpha+i} = [d_1 \otimes f_1, d_{j+2} \otimes f] V_{\alpha_i} = 0 \quad \forall\ f_1 \in I_1,\ f \in I_1^j I_2,
  \]
  and the general result follows by induction.

  We next claim that $I_1^{N^2}I_2 \subseteq I_j$ for all $j \ge 1$.  The result is clear for $j=1,2$, so we assume $j \ge 3$.  Consider the chain of subspaces
  \[
    A/I_j \supseteq (I_2 + I_j)/I_j \supseteq (I_1I_2 + I_j)/I_j \supseteq (I_1^2I_2 + I_j)/I_j \supseteq \dots.
  \]
  Since $\dim A/I_j \le N^2$, this chain must stabilize and so we have $I_1^mI_2 + I_j = I_1^{m+1}I_2 + I_j$ for some $m \le N^2$.  This implies that $I_1^\ell I_2 + I_j = I_1^m I_2 + I_j$ for all $\ell \ge m$.  Now, by the above, we have $I_1^{j-2}I_2 \subseteq I_j$, which implies that $I_1^\ell I_2 + I_j = I_j$ for sufficiently large $\ell$.  Thus $I_1^m I_2 + I_j = I_j$, i.e., $I_1^m I_2 \subseteq I_j$, and so $I_1^{N^2} I_2 \subseteq I_j$ as desired.

  Arguments analogous to those given above show that $I_{-1}^{N^2}I_{-2} \subseteq I_{-j}$ for all $j \ge 1$.  It follows that
  \begin{equation} \label{eq:J-def}
    J := I_{-1}^{N^2}I_{-2}I_1^{N^2}I_2 \subseteq I_j \quad \forall\ j \in \Z \setminus \{0\}.
  \end{equation}
  Note that $J$ has finite codimension in $A$ since $I_{-1}, I_{-2}, I_1, I_2$ do.  Now, by definition, any element $f \in J$ can be written as a sum of elements of the form $f_{-1} f_1$ and as a sum of elements of the form $f_{-2}f_2$ for $f_j \in I_j$, $j \in \{\pm 1, \pm 2\}$.  Since
  \begin{gather*}
    2d_0 \otimes f_{-1}f_1 = [d_{-1} \otimes f_{-1}, d_1 \otimes f_1] \quad \text{and} \\
    4d_0 \otimes f_{-2}f_2 - (c \otimes f_{-2}f_2)/2 = [d_{-2} \otimes f_{-2}, d_2 \otimes f_2]
  \end{gather*}
  act as zero on $V_{\alpha+i}$, it follows that $d_0 \otimes J$ and $c \otimes J$ annihilate $V_{\alpha+i}$.  Combined with~\eqref{eq:J-def}, this gives that $(\Vir \otimes J) V_{\alpha+i} = 0$.

  Since $V_{\alpha + i} \ne 0$ and $V$ is irreducible, we have $U(\cV) V_{\alpha+i} = V$.  To show that $(\Vir \otimes J)V = 0$, it therefore suffices to show that $(\Vir \otimes J)U_n(\cV) V_{\alpha+i}=0$ for all $n \in \N$.  We do this by induction, the case $n=0$ having been proven above.  Assume the result is true for $k < n$.  An arbitrary element of $U_n(\cV) V_{\alpha+i}$ can be written as a sum of elements of the form
  \begin{gather*}
    (u_1 \otimes f_1) \cdots (u_s \otimes f_s) v_{\alpha+i},\\ \text{where} \quad s \le n,\ v_{\alpha+i} \in V_{\alpha + i},\ u_j \in \Vir,\ f_j \in A,\ j=1,\dots,s.
  \end{gather*}
  For $u \in \Vir$ and $f \in J$, we have
  \begin{multline*}
    (u \otimes f) (u_1 \otimes f_1) \cdots (u_s \otimes f_s) v_{\alpha+i} \\
    = \sum_{j=1}^s (u_1 \otimes f_1) \cdots (u_{j-1} \otimes f_{j-1}) ([u,u_j] \otimes ff_j) (u_{j+1} \otimes f_{j+1}) \cdots (u_s \otimes f_s) v_{\alpha+i}=0,
  \end{multline*}
  where in the last equality we used the induction hypothesis.  It follows that we have $(\Vir \otimes J) V = 0$ as desired.\qed
\end{proof}

\begin{proposition} \label{prop:UBM-have-single-point-support}
  Suppose $V$ is a uniformly bounded irreducible $\cV$-module.  Then $(\Vir \otimes J) V = 0$ for some ideal $J \trianglelefteq A$ of finite codimension with $J$ supported at a single point (i.e., $\rad J$ is a maximal ideal of $A$).  In other words, the nontrivial uniformly bounded irreducible $\cV$-modules have support at a single point.
\end{proposition}

\begin{proof}
  The result is clear if $V$ is trivial and so we assume it is nontrivial.  By Proposition~\ref{prop:irred-killed-by-finite-codim-ideal}, there exists an ideal $J \trianglelefteq A$ of finite codimension such that we have $(\Vir \otimes J) V = 0$.  Since $J$ has finite codimension, we may write $J = J_1 J_2 \dots J_\ell$ for ideals $J_1, \dots, J_\ell$ supported at distinct points.  Now, the action of $\Vir \otimes A$ on $V$ factors through
  \[
    (\Vir \otimes A)/(\Vir \otimes J) \cong (\Vir \otimes A/J) \cong \Vir \otimes ((A/J_1) \oplus \dots \oplus (A/J_\ell)) \cong \bigoplus_{i=1}^\ell (\Vir \otimes A/J_i).
  \]
  It suffices to show that at most one summand above acts nontrivially on $V$.  Without loss of generality, assume the first summand $L_1 := \Vir \otimes A/J_1$ acts nontrivially.  Define $L_2 = \bigoplus_{i=2}^\ell (\Vir \otimes A/J_i)$, $L=L_1 \oplus L_2$, and let
  \begin{gather*}
    \delta_1 = (d_0 \otimes (1 + J_1)) \in L_1,\quad \delta_2 = (0,d_0 \otimes (1 + J_2), \dots, d_0 \otimes (1+J_n)) \in L_2, \\
    \delta = \delta_1 + \delta_2.
  \end{gather*}
  Note that $d_0 v = \delta v$ for all $v \in V$ and that the actions of $\delta_1, \delta_2, \delta$ commute.  It follows that, for $i=1,2$, $\delta_i$ preserves the finite-dimensional $d_0$-eigenspaces.  Therefore $\delta_i$ has an eigenvector $v \in V$.  Since the action of $\delta_i$ on $L$ is diagonalizable and $v$ generates $V$ as a module over $L$, we see that $\delta_i$ acts diagonalizably on $V$ for $i=1,2$.

  Because the eigenvalues of the action of $d_0$ on $L$ are integers, the above discussions implies that we have a decomposition
  \[
    V = \bigoplus_{j,k \in \Z} V_{(j,k)},\quad V_{(j,k)} = \{v \in V\ |\ \delta_1 v = (\alpha + j)v,\ \delta_2 v = (\beta + k)v\},
  \]
  for some fixed $\alpha, \beta \in \C$.  Since $[L_1,L_2]=0$, for each $k \in \Z$ we have that $V_{(*,k)} := \bigoplus_{j \in \Z} V_{(j,k)}$ is an $L_1$-submodule of $V$.  None of these can be a nonzero trivial module since if $L_1$ acts by zero on any nonzero element $v \in V$, then, since $[L_1,L_2]=0$ and $V$ is irreducible (hence $v$ generates $V$ as an $L$-module and thus as an $L_2$-module), $L_1$ would act trivially on all of $V$ which contradicts our assumption.  Thus, since $V$ is uniformly bounded, by Corollary~\ref{cor:Vir-UBM=int-series} and Lemma~\ref{lem:inter-series-dim-1-wt-spaces} we must have that $V_{(j,k)} \ne 0$ for all $\alpha + j \ne 0$ whenever $V_{(*,k)} \ne 0$.  By an analogous argument, we can assume that $L_2$ acts nontrivially on all $V_{(j,*)}$, $\alpha+j \ne 0$.  It follows that $V_{(j,k)} \ne 0$ whenever $\alpha+j \ne 0$ and $\beta + k \ne 0$.  Now,
  \[
    V_{\alpha + \beta} \supseteq \bigoplus_{j \in \Z} V_{j,-j},
  \]
  with the right-hand space being infinite-dimensional.  This contradicts the fact that the weight spaces of $V$ are finite-dimensional, completing the proof.\qed
\end{proof}

\begin{remark}
  Proposition~\ref{prop:UBM-have-single-point-support} shows that the situation for uniformly bounded $\cV$-modules is quite different than for the finite-dimensional modules for $\g \otimes A$ or its equivariant analogue (the equivariant map algebras), when $\g$ is a finite-dimensional algebra.  In the latter case, irreducible modules can be supported at any finite number of points (see \cite{NSS09}).  This is not possible for uniformly bounded $\cV$-modules for the simple reason that a tensor product of two nontrivial uniformly bounded modules will always have infinite-dimensional weight spaces.  However, we will see in Section~\ref{sec:hw-modules} that the highest weight quasifinite $\cV$-modules can have support at more than one point.
\end{remark}

If we have a vector space decomposition $\g \cong W \oplus W'$ of a Lie algebra $\g$, we can pick ordered bases $\cB$ and $\cB'$ of $W$ and $W'$ (respectively) and obtain an ordered basis of $\g$ by declaring $b \ge b'$ for all $b \in \cB$, $b' \in \cB'$.  Then, by the PBW theorem, the set of monomials
\begin{multline*}
  \{x_1 \cdots x_n y_1 \cdots y_m\ |\ n,m \in \N,\ x_1,\dots,x_n \in \cB,\ x_1 \ge \dots \ge x_n,\\
  y_1,\dots,y_m \in \cB',\ y_1 \ge \dots \ge y_m\},
\end{multline*}
forms a basis of $U(\g)$.  By a slight abuse of terminology, we will denote by $U_n(W)$ the subspace of $U(\g)$ spanned by all monomials of the form $x_1 \cdots x_s$, $s \in \N$, $s \le n$, $x_1,\dots,x_s \in \cB$, $x_1 \ge \dots \ge x_s$, and we set $U(W) = \bigcup_n U_n(W)$.  We define $U_n(W')$ and $U(W')$ similarly.  Thus $U(\g) \cong U(W) \otimes U(W')$.  Note that when $W$ is actually a subalgebra of $\g$, $U(W)$ is the usual enveloping algebra of $W$ (and similarly for $W'$).

\begin{lemma} \label{lem:ab-ideal-commutation}
  Suppose $\fa$ is an abelian ideal of a Lie algebra $\g$ and fix a vector space decomposition $\g = W \oplus \fa$ so that $U(\g) \cong U(W) \otimes U(\fa)$.  Then, for all $n \in \N_+$,
  \begin{align*}
    [\fa,U_n(W) U(\fa)] &\subseteq U_{n-1}(W) U(\fa), \quad \text{and} \\
    \fa (U_n(W) U(\fa)) &\subseteq U_n(W) U(\fa).
  \end{align*}
\end{lemma}

\begin{proof}
  Since the second inclusion follows easily from the first, we prove only the first, by induction on $n$.  The case $n=1$ follows immediately from the fact that for all $a \in \fa$, $w \in W$, $u \in U(\fa)$, we have
  \[
    [a,wu] = [a,w]u + w[a,u] = [a,w]u \in U(\fa),
  \]
  where we have used that $U(\fa)$ is commutative since $\fa$ is abelian.

  Now suppose $n > 1$.  The space $U_n(W) U(\fa)$ is spanned by elements of the form $w_1 \cdots w_s u$, where $s \le n$, $w_i \in W$ for $i =1,\dots,s$, $u \in U(\fa)$.  If $a \in \fa$, then
  \begin{align*}
    [a,w_1 \cdots w_s u] &= [a,w_1] w_2 \cdots w_s u + w_1 [a,w_2 \cdots w_s u] \\
    &= [[a,w_1],w_2 \cdots w_s]u + w_2 \cdots w_s [a,w_1] u + w_1 [a,w_2 \cdots w_s u].
  \end{align*}
  Now, $[a,w_1] \in \fa$ since $\fa$ is an ideal.  Therefore $[[a,w_1],w_2 \cdots w_s]u \in U_{s-2}(W) U(\fa) \linebreak \subseteq U_{n-1}(W) U(\fa)$ by the induction hypothesis.  In addition, $w_2 \cdots w_s [a,w_1] u \in U_{s-1}(W) U(\fa) \subseteq U_{n-1} U(\fa)$.  Finally, $[a,w_2 \cdots w_s u] \in U_{s-2}(W) U(\fa)$ by the induction hypothesis, and so $w_1 [a,w_2 \cdots w_s u] \in U_{s-1}(W) U(\fa) \subseteq U_{n-1}(W) U(\fa)$.  This completes the proof.\qed
\end{proof}

\begin{proposition} \label{prop:J2=0-kills-V}
  Suppose $V$ is a uniformly bounded irreducible $\cV$-module, with $A$ finite-dimensional.  Then $(\Vir \otimes J) V=0$ for any ideal $J \trianglelefteq A$ satisfying $J^2=0$.
\end{proposition}

\begin{proof}
  We may assume that $V$ is nontrivial since otherwise the statement is clear.  Let $J$ be an ideal of $A$ such that $J^2=0$.  We have a weight decomposition $V = \bigoplus_{i \in \Z} V_{\alpha + i}$ for some $\alpha \in \C$.  Fix $i \in \Z$ such that $V_{\alpha+i} \ne 0$ and let $f \in J$.  Since the operator $d_0 \otimes f$ fixes the finite-dimensional vector space $V_{\alpha+i}$, it has an eigenvector.  In other words, there exists a nonzero $v \in V_{\alpha+i}$ and $a \in \C$ such that $(d_0 \otimes f) v = av$.  We split the proof into the following steps:
  \begin{description}
    \item[Step 1:] Show that $(d_0 \otimes f)-a$ acts nilpotently on $V$.
    \item[Step 2:] Show that $a=0$ and $c \otimes J$ acts by zero on $V$.
    \item[Step 3:] Show that $(d_{i_1} \otimes f) \cdots (d_{i_N} \otimes f) V = 0$ for all $i_1,\dots,i_N \in \Z,\ f \in J$.
    \item[Step 4:] Show that $(\Vir \otimes J)V=0$.
  \end{description}

  \emph{Step 1:} We first show that $(d_0 \otimes f) -a$ acts locally nilpotently on $V$.  Pick a vector space complement $B$ to $J$ in $A$.  So $A = B \oplus J$ as vector spaces.  Then we have the vector space decomposition $\cV = (\Vir' \otimes B) \oplus (\C c \otimes B) \oplus (\Vir \otimes J)$, where $\Vir' := \bigoplus_{m \in \Z} d_m$.  We therefore have, by the PBW Theorem,
  \[
    U(\Vir \otimes A) \cong U(\Vir' \otimes B) \otimes U((\C c \otimes B) \oplus (\Vir \otimes J)).
  \]
  Note that since $J^2=0$ and $c$ is central in $\Vir$, $\tilde U := U((\C c \otimes B) \oplus (\Vir \otimes J))$ is a commutative associative algebra.  Since $V$ is irreducible, we have $V = U(\cV) v$.  Thus our claim is equivalent to proving that $((d_0 \otimes f) - a)^{n+1}$ acts by zero on $U_n(\Vir' \otimes B) \tilde U v$ for all $n \in \N$.  We prove this by induction.  The case $n=0$ follows immediately from the commutativity of $\tilde U$ and the fact that $(d_0 \otimes f) - a$ annihilates $v$.  Now consider $n \ge 1$.  For $s \le n$, $u_1, \dots, u_s \in \Vir' \otimes B$, and $u \in \tilde U$, we have
  \begin{align*}
    ((d_0 \otimes f) - a)^{n+1} u_1 \cdots u_s u  v
    &= ((d_0 \otimes f) -a)^n [(d_0 \otimes f) - a, u_1 \cdots u_s u] v \\
    &= ((d_0 \otimes f) -a)^n [(d_0 \otimes f), u_1 \cdots u_s u] v.
  \end{align*}
  By Lemma~\ref{lem:ab-ideal-commutation},
  \[
    [(d_0 \otimes f), u_1 \cdots u_s u] \in U_{s-1}(\Vir' \otimes B) \tilde U \subseteq U_{n-1}(\Vir' \otimes B) \tilde U,
  \]
  and so $((d_0 \otimes f) -a)^n [(d_0 \otimes f), u_1 \cdots u_s u] v = 0$ by the induction hypothesis. This completes the proof that $(d_0 \otimes f) - a$ acts locally nilpotently on $V$.

  Since $V$ is uniformly bounded, we can choose $N \in \N$ such that $\dim V_{\alpha + i} \le N$ for all $i \in \Z$. Thus $(d_0 \otimes f) - a$ acts nilpotently on $V_{\alpha + i}$ and, in fact, $((d_0 \otimes f)-a)^N V_{\alpha + i} = 0$ for all $i \in \Z$.  Therefore
  \begin{equation} \label{eq:d0f-a-nilpotence}
    ((d_0 \otimes f)-a)^N V = 0.
  \end{equation}
  Thus $(d_0 \otimes f) - a$ acts nilpotently on $V$.

  \medskip

  \emph{Step 2:} Since $\Vir \otimes J$ is abelian, we have
  \[
    [d_j, ((d_0 \otimes f)-a)^m] = m((d_0 \otimes f)-a)^{m-1}[d_j,(d_0 \otimes f)-a] \quad \forall\ m \in \N_+,
  \]
  and so
  \begin{equation} \label{eq:dj-d0-identity}
    [d_j, ((d_0 \otimes f)-a)^m] = -jm(d_j \otimes f)((d_0 \otimes f)-a)^{m-1} \quad \forall\ m \in \N_+.
  \end{equation}

  From~\eqref{eq:d0f-a-nilpotence} and~\eqref{eq:dj-d0-identity}, it follows by an easy induction that
  \begin{equation} \label{eq:dj-d0-kill-V}
    (d_j \otimes f)^r (d_0 \otimes f - a)^{N-r} V = 0 \quad \text{for all } j \in \Z \setminus \{0\},\ 0 \le r \le N.
  \end{equation}

  Since $c \otimes f$ is central, it acts by some scalar $c' \in \C$ on the irreducible module $V$ by Schur's Lemma.  We want to show that $a=c'=0$.  Suppose, on the contrary, that $a \ne 0$ or $c' \ne 0$.  Then we can choose $j \in \Z \setminus \{0\}$ such that
  \begin{equation} \label{eq:nonzero-choices}
    2ja - \frac{j^3-j}{12}c' \ne 0.
  \end{equation}

  Taking $r=N$ in \eqref{eq:dj-d0-kill-V}, we see that $(d_j \otimes f)^N V=0$.  Let $m$ be the minimal element of $\N$ such that $(d_j \otimes f)^m V = 0$ (so, clearly, $1
  \le m \le N$).  Since $\Vir \otimes J$ is abelian, we have
  \begin{align*}
    0 &= [d_{-j}, (d_j \otimes f)^m] V \\
    &= m(d_j \otimes f)^{m-1} [d_{-j}, d_j \otimes f] V \\
    &= m(d_j \otimes f)^{m-1} \left( 2j d_0 \otimes f - \frac{j^3-j}{12} c \otimes f \right) V.
  \end{align*}
  For each $i \in \Z$, by~\eqref{eq:nonzero-choices}, $(2jd_0 \otimes f - \frac{j^3-j}{12} c \otimes f)$ acts invertibly on the generalized $(d_0 \otimes f)$-eigenspace of $V_{\alpha+i}$ corresponding to the eigenvalue $a$.  Thus, we see from the above that $(d_j \otimes f)^{m-1}$ acts by zero on such generalized eigenspaces.  On the other hand, we have from~\eqref{eq:dj-d0-kill-V} that
  \[
    (d_j \otimes f)^{m-1} ((d_0 \otimes f)-a)^{N-m+1}V = 0,
  \]
  which implies that $(d_j \otimes f)^{m-1}$ also acts by zero on all the generalized eigenspaces of $V_{\alpha+i}$, $i \in \Z$, corresponding to any eigenvalue not equal to $a$.  It follows that $(d_j \otimes V)^{m-1} V = 0$, contradicting the choice of $m$.  Therefore $a=c'=0$.

  \medskip

  \emph{Step 3:} Since the above arguments hold for arbitrary $f \in J$, we have ${(d_0 \otimes J)^N V} \linebreak =0$.  We show that
  \begin{equation} \label{eq:Vir-J-nilpotent-elements}
    (d_{i_1} \otimes f) \cdots (d_{i_r} \otimes f) (d_0 \otimes f)^{N-r} V = 0 \quad \forall\ 0 \le r \le N,\ i_1, \dots, i_r \in \Z \setminus \{0\},\ f \in J.
  \end{equation}
  We have already proved the base case $r=0$.  Now assume the result holds for some $0 \le r < N$.  Then, for $i_1,\dots,i_{r+1} \in \Z \setminus \{0\}$, $f \in J$,
  \begin{align*}
    0 &= (d_{i_1} \otimes f) \cdots (d_{i_r} \otimes f) (d_0 \otimes f)^{N-r} V \\
    &= d_{i_{r+1}} (d_{i_1} \otimes f) \cdots (d_{i_r} \otimes f) (d_0 \otimes f)^{N-r} V \\
    &= \sum_{k=1}^r [d_{i_{r+1}}, d_{i_k} \otimes f] (d_{i_1} \otimes f) \cdots (d_{i_{k-1}} \otimes f)(d_{i_{k+1}} \otimes f) \cdots (d_{i_r} \otimes f) (d_0 \otimes f)^{N-r} V \\
    &\qquad + (N-r) [d_{i_{r+1}}, d_0 \otimes f] (d_{i_1} \otimes f) \cdots (d_{i_r} \otimes f) (d_0 \otimes f)^{N-r-1} V \\
    &\qquad + (d_{i_1} \otimes f) \cdots (d_{i_r} \otimes f) (d_0 \otimes f)^{N-r} d_{i_{r+1}} V \\
    &= (i_k-i_{r+1}) \sum_{k=1}^r (d_{i_{r+1} + i_k} \otimes f) (d_{i_1} \otimes f) \\
    &\qquad \qquad \qquad \qquad \qquad \cdots (d_{i_{k-1}} \otimes f)(d_{i_{k+1}} \otimes f) \cdots (d_{i_r} \otimes f) (d_0 \otimes f)^{N-r} V \\
    &\qquad -i_{r+1}(N-r) (d_{i_{r+1}} \otimes f) (d_{i_1} \otimes f) \cdots (d_{i_r} \otimes f) (d_0 \otimes f)^{N-r-1} V \\
    &= -i_{r+1}(N-r) (d_{i_1} \otimes f) \cdots (d_{i_{r+1}} \otimes f) (d_0 \otimes f)^{N-r-1} V,
  \end{align*}
  where in the fourth equality we have used the fact that $c \otimes J$ acts by zero on $V$.  This completes the inductive step.  Now,~\eqref{eq:Vir-J-nilpotent-elements} immediately implies that
  \begin{equation} \label{eq:di-fixed-f-nilpotence}
    (d_{i_1} \otimes f) \cdots (d_{i_N} \otimes f) V = 0 \quad \text{for all } i_1,\dots,i_N \in \Z,\ f \in J.
  \end{equation}

  \medskip

  \emph{Step 4:} By assumption, $A$ is finite-dimensional.  Let $M = (\dim A)(N-1)+1$.  By expanding in a basis for $A$ and using~\eqref{eq:di-fixed-f-nilpotence}, we see that
  \[
    (d_{i_1} \otimes f_1) \cdots (d_{i_M} \otimes f_M) V = 0 \quad \text{for all } i_1,\dots, i_M \in \Z,\ f_1,\dots,f_M \in J.
  \]
  In other words, $(\Vir \otimes J)^M V = 0$, where the $M$-th power here is interpreted as taking place inside $U(\Vir \otimes J)$.  Thus $U(\cV) (\Vir \otimes J)^M U(\cV) V = 0$.  We claim that
  \begin{equation} \label{eq:universal-M-power}
    \big( U(\cV) (\Vir \otimes J) U(\cV) \big)^M = U(\cV) (\Vir \otimes J)^M U(\cV).
  \end{equation}
  Since $1 \in U(\cV)$, it is clear that the right-hand side of~\eqref{eq:universal-M-power} is contained in the left-hand side.  To prove the reverse inclusion, one expands the left-hand side and uses the Lie bracket (and the fact that $\Vir \otimes J$ is an ideal of $\cV$) to move the elements of $U(\cV)$ to the right or left.  Thus $\big( U(\cV) (\Vir \otimes J) U(\cV) \big)^M V=0$.  This implies that $\big( U(\cV) (\Vir \otimes J) U(\cV) \big) V \ne V$.  Since $V$ is irreducible and $\big( U(\cV) (\Vir \otimes J) U(\cV) \big) V$ is a submodule of $V$, this implies that $\big( U(\cV) (\Vir \otimes J) U(\cV) \big) V = 0$, which in turn implies that $(\Vir \otimes J) V = 0$ as desired.\qed
\end{proof}

\begin{corollary} \label{cor:nilp-ideals-act-by-zero}
  Suppose $V$ is a uniformly bounded irreducible $\cV$-module, with $A$ finite-dimensional.  Then $(\Vir \otimes J) V=0$ for any nilpotent ideal $J$ of $A$.
\end{corollary}

\begin{proof}
  We may assume that $V$ is nontrivial since otherwise the statement is clear.  Let $J$ be a nilpotent ideal of $A$, so that $J^r = 0$ for some $r \in \N_+$.  Choose the minimal $n \in \N_+$ with the property that $(\Vir \otimes J^n) V = 0$.  Suppose $n > 1$.  The action of $\cV$ factors through $\cV/(\Vir \otimes J^n) \cong \Vir \otimes (J/J^n)$, and so we can consider $V$ as a module for this quotient.  Then, by Proposition~\ref{prop:J2=0-kills-V}, we have that $(\Vir \otimes (J^{n-1}/J^n)) V = 0$.  This implies $(\Vir \otimes J^{n-1}) V = 0$, contradicting the choice of $n$.  It follows that $n=1$ and so $(\Vir \otimes J) V = 0$.\qed
\end{proof}

\begin{theorem} \label{thm:UBM-classification}
  Any uniformly bounded irreducible $\cV$-module is a single point evaluation module $\ev_\fm V$ for some maximal ideal $\fm \trianglelefteq A$ and $\Vir$-module $V$ of the intermediate series.
\end{theorem}

\begin{proof}
  It suffices to show that $V$ is annihilated by $\Vir \otimes \fm$ for some maximal ideal $\fm \trianglelefteq A$.  By Proposition~\ref{prop:UBM-have-single-point-support}, there exists an ideal $J \trianglelefteq A$ of finite codimension, with $\fm := \rad J$ a maximal ideal, such that $(\Vir \otimes J) V = 0$.  We can consider $V$ as a module for $(\Vir \otimes A)/(\Vir \otimes J) \cong \Vir \otimes (A/J)$, where the algebra $A/J$ is finite-dimensional.  Since every ideal in a Noetherian ring contains a power of its radical (see, for example, \cite[Prop.~7.14]{AM69}), we have $\fm^r \subseteq J$ for some $r \in \N_+$.  Then $(\fm/J)^r = 0$ in $A/J$, and it follows from Corollary~\ref{cor:nilp-ideals-act-by-zero} that $(\Vir \otimes \fm) V = 0$.\qed
\end{proof}

%
\section{Highest weight modules} \label{sec:hw-modules}
%

In this section we give a classification of the irreducible highest weight quasifinite $\cV$-modules.  We show that they are all tensor products of generalized single point evaluation modules.  In the case $A = \C[t,t^{-1}]$, this was proved in \cite[Theorem~6.4]{GLZ08}.

\begin{proposition} \label{prop:hw-HC-finite-support}
  The irreducible highest weight module $V(\varphi)$, $\varphi \in \hom_\C(\cV_0,\C)$, is a quasifinite module if and only if there exists an ideal $J \trianglelefteq A$ of finite codimension such that $\varphi(\Vir_0 \otimes J) = 0$ and, in this case, $(\Vir \otimes J)V(\varphi)=0$.  In particular, an irreducible highest weight module is a quasifinite module if and only if it has finite support.
\end{proposition}

\begin{proof}
  Let $J$ denote the kernel of the linear map
  \[
    A \to V(\varphi)_{\varphi(d_0)-2},\quad f \mapsto (d_{-2} \otimes f) v_\varphi,\ f \in A.
  \]
  We claim $J$ is an ideal of $A$.  Clearly $J$ is a linear subspace of $A$.  For $f \in J$, $g \in A$, we have
  \[
    0 = [d_0 \otimes g, d_{-2} \otimes f] v_\varphi = -2(d_{-2} \otimes gf) v_\varphi,
  \]
  which implies $gf \in J$.  In the above, we have used the fact that $d_0 \otimes g$ preserves the weight space $V_{\varphi(d_0)}$, which is spanned by $v_\varphi$.  Next we claim that $\varphi(\Vir_0 \otimes J) = 0$.  Fix $f \in J$.  Then
  \[
    0 = (d_2 \otimes 1)(d_{-2} \otimes f) v_\varphi = [d_2 \otimes 1, d_{-2} \otimes f] v_\varphi = \left( \left(-4d_0 + \frac{1}{2}c \right) \otimes f \right) v_\varphi,
  \]
  and
  \begin{align*}
    0 &= (d_1 \otimes 1)(d_1 \otimes 1)(d_{-2} \otimes f) v_\varphi = (d_1 \otimes 1)[d_1 \otimes 1, d_{-2} \otimes f] v_\varphi \\
    &= -3(d_1 \otimes g)(d_{-1} \otimes f) v_\varphi = -3[d_1 \otimes 1, d_{-1} \otimes f] = 6(d_0 \otimes f) v_\varphi.
  \end{align*}
  Thus $\varphi(d_0 \otimes f) v_\varphi = (d_0 \otimes f) v_\varphi = 0$ and $\varphi(c \otimes f) v_\varphi = (c \otimes f) v_\varphi = 0$ for all $f \in J$, proving our claim.  If $V(\varphi)$ is a quasifinite module, the weight space $V(\varphi)_{\varphi(d_0)-2}$ is finite-dimensional, and so $J$ has finite codimension in $A$.  This completes the proof of the reverse implication asserted in the proposition.

  Now assume that there exists an ideal $J \trianglelefteq A$ of finite codimension such that $\varphi(\Vir_0 \otimes J) = 0$.  We first show that $(\Vir \otimes J) v_\varphi=0$.  It suffices to show that $(d_n \otimes J) v_\varphi$ for all $n \in \Z$, which we show by induction.  The result holds by definition of $V(\varphi)$ for $n > 0$ and by the assumption on $J$ for $n=0$.  Now assume the result holds for all $n > k$ for some $k \in \Z$. Then for all $f \in J$ and $g \in A$, we have
  \begin{gather*}
    (d_1 \otimes g)(d_k \otimes f) v_\varphi = [d_1 \otimes g, d_k \otimes f] v_\varphi = (k-1)(d_{k+1} \otimes gf) v_\varphi = 0, \\
    (d_2 \otimes g)(d_k \otimes f) v_\varphi = [d_2 \otimes g, d_k \otimes f] v_\varphi = (k-2)\left( \left(d_{k+2} + \delta_{k,-2} \tfrac{1}{2} c \right) \otimes gf \right) v_\varphi = 0.
  \end{gather*}
  Suppose $(d_k \otimes f) v_\varphi \ne 0$.  Since elements of the form $d_1 \otimes g$, $d_2 \otimes g$, $g \in A$, generate $\cV_+$, this would imply that $(d_k \otimes f) v_\varphi$ is a highest weight vector, contradicting the irreducibility of $V(\varphi)$.  Therefore $(d_k \otimes f)v_\varphi=0$, completing the inductive step.

  Now, since $\Vir \otimes J$ is an ideal of $\cV$, the set $W$ of all elements of $V(\varphi)$ annihilated by $\Vir \otimes J$ is a $\cV$-submodule of $V(\varphi)$.  Since $W$ contains $v_\varphi$ by the above, it is nonzero.  Therefore, since $V(\varphi)$ is irreducible, $W=V(\varphi)$.  In other words, $(\Vir \otimes J)V(\varphi)=0$.

  It follows from the above that $V(\varphi)$ can be considered as a module over $\cV/(\Vir \otimes J) \cong \Vir \otimes (A/J)$ and that $V(\varphi) = U(\Vir_- \otimes (A/J)) v_\varphi$.  Since $J$ has finite codimension in $A$, the weight spaces of $U(\Vir_- \otimes (A/J))$ are finite-dimensional.  Hence the same property holds for $V(\varphi)$, which is thus a quasifinite module.\qed
\end{proof}

\begin{corollary}
  If $A$ is finite-dimensional, then all highest or lowest weight $\cV$-modules are quasifinite modules.
\end{corollary}

\begin{proof}
  This follows from the reasoning in the last paragraph of the proof of Proposition~\ref{prop:hw-HC-finite-support}.\qed
\end{proof}

\begin{theorem} \label{thm:hw-HC-prod-single-points}
  Any irreducible highest weight quasifinite $\cV$-module is a tensor product of irreducible (generalized evaluation) highest weight quasifinite modules supported at single points.
\end{theorem}

\begin{proof}
  Suppose $V(\varphi)$ is an irreducible highest weight quasifinite module.  Then, by Proposition~\ref{prop:hw-HC-finite-support}, $J:= \Ann_A V$ has finite support.  Therefore $\rad J = \fm_1 \cdots \fm_r$ for some distinct maximal ideals $\fm_1, \dots, \fm_r \trianglelefteq A$.  Since every ideal in a Noetherian ring contains a power of its radical (see, for example, \cite[Prop.~7.14]{AM69}), there exists $N \in \N_+$ such that $\fm_1^N \cdots \fm_r^N \subseteq J$.  Then $\varphi(\Vir_0 \otimes \fm_1^N \cdots \fm_r^N)=0$, and so $\varphi$ corresponds to a unique element
  \[
    \bar \varphi \in \left(\Vir_0 \otimes A/\fm_1^N \cdots \fm_r^N \right)^* \cong \bigoplus_{i=1}^r (\Vir_0 \otimes A/\fm_i^N)^*.
  \]
  Let $(\bar \varphi_1, \dots, \bar \varphi_r) \in \bigoplus_{i=1}^r (\Vir_0 \otimes A/\fm_i^N)^*$ be the element corresponding to $\bar \varphi$ under the above isomorphism. For each $1 \le i \le r$, let $\varphi_i$ be the unique element of $(\cV_0)^*$ corresponding to $(0,\dots,0,\bar \varphi_i,0,\dots,0)$ (with the term $\bar \varphi_i$ occurring in the $i$-th position).  We thus have $\varphi = \sum_{i=1}^r \varphi_i$ and $V(\varphi_i)$ has support in the single point corresponding to the maximal ideal $\fm_i$.  Now, the tensor product $\bigotimes_{i=1}^r V(\varphi_i)$ is a weight module with a highest weight vector $v := v_{\varphi_1} \otimes \cdots \otimes v_{\varphi_r}$ and $u v = \varphi(u)v$ for all $u \in \cV_0$.  Since each $V(\varphi_i)$ is absolutely reducible (being irreducible of countable dimension), so is $\bigotimes_{i=1}^r V(\varphi_i)$ (see, for example, \cite[\S7.4, Theorem~2]{Bou58} or \cite[Lemma~2.7]{Li04}).  It follows that $\bigotimes_{i=1}^r V(\varphi_i) \cong V(\varphi)$.\qed
\end{proof}

\begin{corollary}
  If $V(\varphi_1), \dots, V(\varphi_r)$ are irreducible highest weight quasifinite modules with pairwise disjoint supports, then $\bigotimes_{i=1}^r V(\varphi_i) = V(\varphi_1 + \cdots + \varphi_r)$.
\end{corollary}

\begin{proof}
  This follows from the proof of Theorem~\ref{thm:hw-HC-prod-single-points}.\qed
\end{proof}

Combining Proposition~\ref{prop:wt-spaces} and Theorems~\ref{thm:UBM-classification} and~\ref{thm:hw-HC-prod-single-points} yields the following.

\begin{theorem} \label{thm:main-theorem}
  Any irreducible quasifinite $\cV$-module is one of the following:
  \begin{enumerate}
    \item a single point evaluation module corresponding to a $\Vir$-module of the intermediate series (or tensor density module),
    \item a finite tensor product of single point generalized evaluation modules corresponding to irreducible highest weight modules, or
    \item a finite tensor product of single point generalized evaluation modules corresponding to irreducible lowest weight modules.
  \end{enumerate}
  In particular, they are all tensor products of single point generalized evaluation modules.
\end{theorem}

%
\section{Reducibility of Verma modules} \label{sec:Verma-reducibility}
%

In this section, we give a sufficient condition for a Verma module for $\cV$ to be reducible.  This condition is also necessary if $A$ is an infinite-dimensional integral domain.  In the case that $A=\C[t,t^{-1}]$, the condition reduces to the one in \cite[Theorem~6.5]{GLZ08}.

Choose a basis $\cB_A$ of $A$ along with an order $\succ$ on $\cB_A$.  We then have an ordered basis of $\cV_-$ given by
\[
  \{d_{-n} \otimes f \ |\ n \in \N_+,\ f \in \cB_A\},\quad d_{-n_1} \otimes f_1 \succ d_{-n_2} \otimes f_2 \iff (n_1,f_1) \succ (n_2,f_2),
\]
where on the right-hand side we use the usual ordering on $\N_+$ and the lexicographic ordering on pairs.  This induces a PBW basis $\cB$ of $U(\cV_-)$.  We have a natural decomposition $\cB = \bigsqcup_{n=0}^\infty \cB^n$, where
\begin{multline*}
  \cB^n = \{(d_{-i_1} \otimes f_1) \cdots (d_{-i_n} \otimes f_n) \ |\ i_1,\dots,i_n \in \N_+,\ f_1,\dots,f_n \in \cB_A,\\
  (i_1,f_1) \succ \dots \succ (i_n,f_n) \}.
\end{multline*}
Note that, here and in what follows, we always write elements of $\cB$ with the factors in decreasing order.  We write $\height X = n$ for $X \in \cB^n$.  Define an ordering on $\cB$ by setting
\begin{align*}
  (d_{-i_1} \otimes f_1) \cdots (d_{-i_r} \otimes f_r) &\succ (d_{-j_1} \otimes g_1) \cdots (d_{-j_s} \otimes g_s) \\
  \iff (r,i_1,\dots,i_r,f_1,\dots,f_r) &\succ (s,j_1,\dots,j_s,g_1,\dots,g_s),
\end{align*}
where we again use the lexicographic ordering on tuples.

For $n,m \in \Z$, set $U^n_{-m} = U_n(\cV_-)_{-m}$, where we remind the reader that here $n$ refers to the natural filtration on the enveloping algebra and $-m$ denotes the weight (corresponding to the eigenvalue of the action of $d_0$).  Thus
\[
  U^{n_1}_{-m_1} U^{n_2}_{-m_2} \subseteq U^{n_1+n_2}_{-(m_1+m_2)} \quad \text{for all } n_1,n_2,m_1,m_2 \in \N.
\]
In particular,
\[
  (d_{-i_1} \otimes f_1) \cdots (d_{-i_n} \otimes f_n) \in U^n_{-(i_1 + \cdots + i_n)} \quad \text{for all } n,i_1,\dots,i_n \in \N_+,\ f_1,\dots,f_n \in A.
\]

Any element $X \in U(\cV_-)$ can be written as $\sum_{i=1}^n a_iX_i$ for $a_i \in \C$ and $X_1, \dots, X_n \in \cB$ with $X_1 \succ \cdots \succ X_n$.  We define
\[
  \height X = \height X_1,\quad \hm X = a_1X_1
\]
(here $\hm$ stands for \emph{highest term}).  By convention, we set $\height 0 = -1$ and $\hm 0 = 0$.  By definition, $\cB v_\varphi := \{b v_\varphi\ |\ v \in \cB\}$ is a basis for $M(\varphi)$.  For elements of this basis we define
\[
  \height (Xv_\varphi) = \height X,\quad \hm (Xv_\varphi) = (\hm X) v_\varphi.
\]

For a set of indeterminates $\mathcal{X}$, we have the natural grading on $\kk[\mathcal{X}]$, where the elements of $\mathcal{X}$ have degree one.  We thank D.~Daigle for the statement and proof of the following lemma, which will be used in the proof of Theorem~\ref{thm:Verma-reducible}.

\begin{lemma} \label{lem:DD}
  Suppose $R=\kk[\mathcal{X}]$ is a polynomial algebra over a field $\kk$, where $\mathcal{X}$ is an infinite set of indeterminates.  Write $R = \bigoplus_{d \in \N} R_d$, where $R_d$ is the space of homogeneous polynomials of degree $d$.  Let $M_1,\dots,M_p$ be pairwise distinct monomials in $R$, all of the same degree.  Then, for $p \in \N$, the subspace
  \[
    U = \left\{(L_1,\dots,L_p) \in R_1^p\ \left|\ \sum_{i=1}^p L_iM_i=0 \right. \right\}
  \]
  of $R_1^p$ is finite-dimensional over $\kk$.
\end{lemma}

\begin{proof}
  Choose a finite subset $\mathcal{X}'$ of $\mathcal{X}$ such that $M_1,\dots,M_p \in \kk[\mathcal{X}']$, and let $\mathcal{X}'' = \mathcal{X} \setminus \mathcal{X}'$.  Define $R' = \kk[\mathcal{X}'] = \bigoplus_{d \in \N} R'_d$ and $R'' = \kk[\mathcal{X}''] = \bigoplus_{d \in \N} R_d''$, where $R'_d$ and $R''_d$ are the spaces of homogeneous polynomials of degree $d$.  Then $R_1 = R_1' \oplus R_1''$.

  To prove the lemma, it is enough to show that $U \subseteq (R_1')^p$.  Assume the contrary, and consider $(L_1,\dots,L_p) \in U$ such that $(L_1,\dots,L_p) \notin (R_1')^p$.  For each $i \in \{1,\dots,p\}$, write $L_i = L_i' + L_i''$ with $L_i' \in R_1'$ and $L_i'' \in R_1''$.  Then
  \begin{equation} \label{eq:DD1}
    \sum_{i=1}^p L_i' M_i + \sum_{i=1}^p L_i'' M_i = 0,
  \end{equation}
  and moreover $L_{i_0}'' \ne 0$ for some $i_0 \in \{1,\dots,p\}$.

  Let $\varphi : R \to R'$ be the $\kk$-algebra homomorphism that maps each element of $\mathcal{X}'$ to itself and each element of $\mathcal{X}''$ to zero.  Applying $\varphi$ to \eqref{eq:DD1} yields $\sum_{i=1}^p L_i' M_i = 0$, hence
  \begin{equation} \label{eq:DD2}
    \sum_{i=1}^p L_i'' M_i = 0.
  \end{equation}
  Now choose a $\kk$-algebra homomorphism $\psi : R \to R'$ that maps each element of $\mathcal{X}'$ to itself and each element of $\mathcal{X}''$ to an element of $\kk$, in such a way that $\psi(L_{i_0}'') \ne 0$.  Applying $\psi$ to \eqref{eq:DD2} yields $\sum_{i=1}^p \lambda_i M_i = 0$ for some $\lambda_1,\dots,\lambda_p$ not all zero.  Since $M_1,\dots,M_p$ are pairwise distinct, this is a contradiction.\qed
\end{proof}

\begin{theorem} \label{thm:Verma-reducible}
  The Verma module $M(\varphi)$, $\varphi \in \hom_\C(\cV_0,\C)$, is reducible if there exists a nontrivial ideal $J \trianglelefteq A$ such that $\varphi(d_0 \otimes J)=0$.  If $A$ is an infinite-dimensional integral domain, the reverse implication also holds.
\end{theorem}

\begin{proof}
  First suppose there exists a nontrivial ideal $J \trianglelefteq A$ such that $\varphi(d_0 \otimes J)=0$.  For $f \in J$ and $g \in A$, we have
  \[
    (d_1 \otimes g)(d_{-1} \otimes f) \tilde v_\varphi = [d_1 \otimes g, d_{-1} \otimes f] \tilde v_\varphi = -2(d_0 \otimes gf) \tilde v_\varphi = 0.
  \]
  Furthermore, for $m \ge 2$, we have
  \[
    (d_m \otimes g)(d_{-1} \otimes f) \tilde v_\varphi = [d_m \otimes g, d_{-1} \otimes f] \tilde v_\varphi = (-1-m)(d_{m-1} \otimes gf) \tilde v_\varphi = 0.
  \]
  This implies that $(d_{-1} \otimes f) \tilde v_\varphi$ is a highest weight vector and hence $M(\varphi)$ is reducible.

  Now suppose $A$ is an infinite-dimensional integral domain and there is no ideal $J \trianglelefteq A$ such that $\varphi(d_0 \otimes J)=0$.  To prove that $M(\varphi)$ is irreducible, it suffices to show that $M(\varphi)_{-n} = V(\varphi)_{-n}$ for all $n \in \N$.  We prove this by induction, the case $n=0$ being trivial.

  Suppose $M(\varphi)_{-1} \ne V(\varphi)_{-1}$.  Then there exists a nonzero $f \in A$ such that $(d_{-1} \otimes f) v_\varphi = 0$.  Then, for all $g \in A$, we have
  \[
    -2 \varphi(d_0 \otimes gf) v_\varphi = -2(d_0 \otimes gf) v_\varphi = [d_1 \otimes g, d_{-1} \otimes f] v_\varphi = 0.
  \]
  This implies that $\varphi(d_0 \otimes J) = 0$, where $J = Af$ is the ideal generated by $f$.  This contradiction implies that $M(\varphi)_{-1} = V(\varphi)_{-1}$.

  Now suppose $n>1$ and $M(\varphi)_{-k} = V(\varphi)_{-k}$ for all $0 \le k < n$.  It suffices to show that $Xv_\varphi \ne 0$ for all $X \in U(\cV_-)_{-n}$.  Towards a contradiction, suppose $Xv_\varphi = 0$ for some $X \in U(\cV_-)_{-n}$, and write $X = \sum_{i=1}^\ell a_i X_i$ for $a_1,\dots,a_\ell \in \C$ and $X_1, \dots, X_\ell \in \cB$ with $X_1 \succ \cdots \succ X_\ell$.  First suppose that $\height X < n$.  Then
  \[
    X_1 = (d_{-i_1} \otimes f_1) \cdots (d_{-i_r} \otimes f_r) (d_{-1} \otimes g_1) \cdots (d_{-1} \otimes g_s)
  \]
  for some $r > 0$ and $i_1 \ge 2$.  Then
  \begin{align*}
    \hm &((d_1 \otimes 1)Xv_\varphi) \\ &= \hm ([d_1 \otimes 1,X] v_\varphi) \\
    &= (-i_r-1)m(d_{-i_1+1} \otimes f_1) \cdots (d_{-i_{r-1}} \otimes f_{r-1}) (d_{-i_r} \otimes f_r) (d_{-1} \otimes g_1) \cdots (d_{-1} \otimes g_s) v_\varphi \\
    &\ne 0,
  \end{align*}
  where $m$ is the number of $(i_k,f_k)$, $1 \le k \le r$, equal to $(i_1,f_1)$, and the fact that the term is nonzero follows from the induction hypothesis.  Thus $X v_\varphi \ne 0$ as desired.

  It remains to consider the case $\height X =n$.  Then there exists $1 \le r \le s \le \ell$ such that
  \begin{gather*}
    \height X_i = n \text{ for } 1 \le i \le r,\quad \height X_i = n-1 \text{ for } r+1 \le i \le s,\\
    \height X_i \le n-2 \text{ for } s+1 \le r \le \ell.
  \end{gather*}
  For $1 \le i \le r$, we have
  \[
    X_i = (d_{-1} \otimes f_{i,1}) \cdots (d_{-1} \otimes f_{i,n})
  \]
  for some $f_{i,1},\dots,f_{i,n} \in \cB_A$.  Now, for $g \in A$, we have
  \begin{align*}
    (d_1 \otimes g) X_i v_\varphi &= [d_1 \otimes g, (d_{-1} \otimes f_{i,1}) \cdots (d_{-1} \otimes f_{i,n})] v_\varphi \\
    &= -2 \sum_{j=1}^n (d_{-1} \otimes f_{i,1}) \cdots (d_{-1} \otimes f_{i,j-1}) (d_0 \otimes f_{i,j}g) \\
    &\qquad \qquad \qquad \cdot (d_{-1} \otimes f_{i,j+1}) \cdots (d_{-1} \otimes f_{i,n}) v_\varphi \\
    &= -2 \sum_{j=1}^n (d_{-1} \otimes f_{i,1}) \cdots \widehat{(d_{-1} \otimes f_{i,j})} \cdots (d_{-1} \otimes f_{i,n}) (d_0 \otimes f_{i,j}g) v_\varphi \\
    + 2 \sum_{j=1}^n \sum_{k=j+1}^n & (d_{-1} \otimes f_{i,1}) \cdots \widehat{(d_{-1} \otimes f_{i,j})} \cdots (d_{-1} \otimes f_{i,k-1}) \\
    &\qquad \qquad \qquad \cdot (d_{-1} \otimes f_{i,j}f_{i,k}g) (d_{-1} \otimes f_{i,k+1}) \cdots (d_{-1} \otimes f_{i,n}) v_\varphi \\
    &= -2 \sum_{j=1}^n \varphi(d_0 \otimes f_{i,j}g) (d_{-1} \otimes f_{i,1}) \cdots \widehat{(d_{-1} \otimes f_{i,j})} \cdots (d_{-1} \otimes f_{i,n}) v_\varphi \\
    + 2 \sum_{j=1}^n \sum_{k=j+1}^n & (d_{-1} \otimes f_{i,j}f_{i,k}g) (d_{-1} \otimes f_{i,1}) \\
    &\qquad \qquad \qquad \cdots \widehat{(d_{-1} \otimes f_{i,j})} \cdots \widehat{(d_{-1} \otimes f_{i,k})} \cdots (d_{-1} \otimes f_{i,n}) v_\varphi,
  \end{align*}
  where the $\hat\ $ above a term means that term is omitted and we use the fact that $d_{-1} \otimes A$ is an abelian subalgebra of $\cV$.

  Now, for $r+1 \le i \le s$, we have
  \[
    X_i = (d_{-2} \otimes f_{i,1}) (d_{-1} \otimes f_{i,2}) \cdots (d_{-1} \otimes f_{i,n-1})
  \]
  for some $f_{i,1}, \dots, f_{i,n-1} \in \cB_A$.  Then, for $g \in A$, we have
  \begin{align*}
    (d_1 \otimes g)X_i v_\varphi &= [d_1 \otimes g,X_i] v_\varphi \\
    &\equiv -3(d_{-1} \otimes f_{i,1}g) (d_{-1} \otimes f_{i,2}) \cdots (d_{-1} \otimes f_{i,n-1}) \mod (U^{n-2}_{-n+1} v_\varphi).
  \end{align*}
  Combining the above computations and using the fact that $(d_1 \otimes g)X_iv_\varphi \in U^{n-2}_{-n+1} v_\varphi$ for $s+1 \le i \le \ell$ and $g \in A$, we have
  \begin{align*}
    0 &= (d_1 \otimes g)X v_\varphi = [d_1 \otimes g,X]v_\varphi \\
    & \equiv -2 \sum_{i=1}^r a_i \sum_{j=1}^n \varphi(d_0 \otimes f_{i,j}g) (d_{-1} \otimes f_{i,1}) \cdots \widehat{(d_{-1} \otimes f_{i,j})} \cdots (d_{-1} \otimes f_{i,n}) v_\varphi \\
    & \quad + 2 \sum_{i=1}^r a_i \sum_{j=1}^n \sum_{k=j+1}^n (d_{-1} \otimes f_{i,j}f_{i,k}g) (d_{-1} \otimes f_{i,1}) \cdots \widehat{(d_{-1} \otimes f_{i,j})} \\
    &\qquad \qquad \qquad \qquad \qquad \qquad \qquad \qquad \quad \cdots \widehat{(d_{-1} \otimes f_{i,k})} \cdots (d_{-1} \otimes f_{i,n}) v_\varphi \\
    & \quad - 3 \sum_{i=r+1}^s a_i (d_{-1} \otimes f_{i,1}g) (d_{-1} \otimes f_{i,2}) \cdots (d_{-1} \otimes f_{i,n-1})v_\varphi \mod (U^{n-2}_{-n+1} v_\varphi) \\
    & \equiv -2 \sum_{i=1}^r a_i \sum_{j=1}^n \varphi(d_0 \otimes f_{i,j}g) (d_{-1} \otimes f_{i,1}) \cdots \widehat{(d_{-1} \otimes f_{i,j})} \cdots (d_{-1} \otimes f_{i,n}) v_\varphi \\
    & \quad + \sum_{i=1}^m \gamma_i (d_{-1} \otimes q_{i,1} g)(d_{-1} \otimes q_{i,2}) \cdots (d_{-1} \otimes q_{i,n-1})v_\varphi \mod (U^{n-2}_{-n+1} v_\varphi),
  \end{align*}
  for some $\gamma_1,\dots,\gamma_m \in \C$ and pairwise distinct $(q_{i,1},\dots,q_{i,n-1}) \in (\cB_A)^{n-1}$, $i=1,\dots,m$.  By the induction hypothesis, we actually have equality:
  \begin{multline}
    0 = -2 \sum_{i=1}^r a_i \sum_{j=1}^n \varphi(d_0 \otimes f_{i,j}g) (d_{-1} \otimes f_{i,1}) \cdots \widehat{(d_{-1} \otimes f_{i,j})} \cdots (d_{-1} \otimes f_{i,n}) v_\varphi \label{eq:sum-proof} \\
    + \sum_{i=1}^m \gamma_i (d_{-1} \otimes q_{i,1} g)(d_{-1} \otimes q_{i,2}) \cdots (d_{-1} \otimes q_{i,n-1})v_\varphi.
  \end{multline}

  \medskip

  \textbf{Claim:} We have $\gamma_i=0$ for $i=1,\dots,m$.

  \medskip

  Assuming the claim, it follows from~\eqref{eq:sum-proof} that
  \[
    0 = -2 \sum_{i=1}^r a_i \sum_{j=1}^n \varphi(d_0 \otimes f_{i,j}g) (d_{-1} \otimes f_{i,1}) \cdots \widehat{(d_{-1} \otimes f_{i,j})} \cdots (d_{-1} \otimes f_{i,n}) v_\varphi.
  \]
  The coefficient of $(d_{-1} \otimes f_{1,1}) \cdots (d_{-1} \otimes f_{1,n-1}) v_\varphi$ in the above expression, which must therefore be equal to zero, is
  \[
    -2 \sum_{i \in I} k_i a_i \varphi(d_0 \otimes f_{i,n}g) = \varphi \left( d_0 \otimes g \left(-2 \sum_{i \in I} k_i a_i f_{i,n} \right) \right),
  \]
  where $I = \{i \ |\ 1 \le i \le r,\ (f_{i,1},\dots,f_{i,n-1}) = (f_{1,1}, \dots, f_{1,n-1})\}$, $k_1$ is the number of $q$ such that $f_{1,n} = f_{1,q}$, and $k_i=1$ for $i \ne 1$.  Note that $f_{i,n} \ne f_{j,n}$ for $i,j \in I$, $i \ne j$.  Thus $F := \sum_{i \in I} k_i a_i f_{i,n} \ne 0$.  It follows that $\varphi (d_0 \otimes J)=0$, where $J$ is the nontrivial ideal of $A$ generated by $F$.  This contradiction completes the proof of the theorem.  It thus remains to prove the above claim.

  \medskip

  \textit{Proof of Claim:}  (We thank D.~Daigle for the following proof.) Let $M_1,\dots,M_p$ be the distinct elements of the set $\{(d_{-1} \otimes q_{i,2}) \cdots (d_{-1} \otimes q_{i,n-1}) \ |\ 1 \le i \le m\}$.  Consider the partition $\{E_1,\dots,E_p\}$ of $\{1,\dots,m\}$ obtained by setting $E_t = \{i \ |\ (d_{-1} \otimes q_{i,2}) \cdots (d_{-1} \otimes q_{i,n-1}) = M_t\}$ for $t=1,\dots,p$.  Then
  \begin{align*}
    \sum_{i=1}^m \gamma_i (d_{-1} \otimes q_{i,1} g)&(d_{-1} \otimes q_{i,2}) \cdots (d_{-1} \otimes q_{i,n-1})v_\varphi \\
    &= \sum_{t=1}^p \sum_{i \in E_t} \gamma_i (d_{-1} \otimes q_{i,1} g)(d_{-1} \otimes q_{i,2}) \cdots (d_{-1} \otimes q_{i,n-1})v_\varphi \\
    &= \sum_{t=1}^p \sum_{i \in E_t} \gamma_i (d_{-1} \otimes q_{i,1}g) M_t v_\varphi \\
    &= \sum_{t=1}^p \left( \sum_{i \in E_t} \gamma_i(d_{-1} \otimes q_{i,1}g) \right) M_t v_\varphi \\
    &= \sum_{t=1}^p (d_{-1} \otimes \beta_t g) M_t v_\varphi,
  \end{align*}
  where $\beta_t = \sum_{i \in E_t} \gamma_i q_{i,1}$ in $A$.

  Now let $\Phi : A \to \C^{rn}$ be the linear map that sends $g \in A$ to the $r \times n$ matrix $(\varphi(d_0 \otimes f_{i,j}g))_{1 \le i \le r, 1 \le j \le n}$.  Then $W:= \ker \Phi$ is a subspace of $A$ with the property that $A/W$ is finite-dimensional. For each $g \in W$, we have $\sum_{i=1}^m \gamma_i(d_{-1} \otimes q_{i,1}g) (d_{-1} \otimes q_{i,2}) \cdots (d_{-1} \otimes q_{i,n-1}) v_\varphi= 0$ by \eqref{eq:sum-proof}, so
  \[
    \sum_{t=1}^p (d_{-1} \otimes \beta_tg)M_t v_\varphi = 0,\quad \text{for all } g \in W.
  \]
  Since $\sum_{t=1}^p (d_{-1} \otimes \beta_tg)M_t \in U^{n-1}_{-n+1}$, it follows from the inductive hypothesis that
  \[
    \sum_{t=1}^p (d_{-1} \otimes \beta_tg)M_t = 0,\quad \text{for all } g \in W.
  \]
  Now, if we view $U(d_{-1} \otimes A)$ as the polynomial algebra $\C[d_{-1} \otimes \cB_A]$, then each $M_t$ is a monomial of degree $n-2$ and each $d_{-1} \otimes \beta_t g$ is a polynomial of degree one.  Thus, by Lemma~\ref{lem:DD}, $\{d_{-1} \otimes \beta_1g,\dots, d_{-1} \otimes \beta_pg\ |\ g \in W\}$ is a finite-dimensional subspace of $(d_{-1} \otimes A)^p$.  Hence $\{(\beta_1g,\dots,\beta_pg)\ |\ g \in W\}$ is a finite-dimensional subspace of $A^p$.  Let $t \in \{1,\dots,p\}$.  Then $\{\beta_tg \ |\ g \in W\}$ is a finite-dimensional subspace of $A$.  Since $A/W$ is finite-dimensional, it follows that the principal ideal $\beta_tA$ of $A$ is finite-dimensional.  Since $A$ is an infinite-dimensional integral domain, this implies that $\beta_t=0$.  Since the $(q_{i,1},\dots,q_{i,n-1}) \in (\cB_A)^{n-1}$, $i=1,\dots,m$, are pairwise distinct, the map $i \mapsto q_{i,1}$, from $E_t$ to $\cB_A$, is injective.  Consequently, the family $(q_{i,1})_{i \in E_t}$ is linearly independent.  Since $\sum_{i \in E_t} \gamma_i q_{i,1}=0$, it follows that $\gamma_i=0$ for all $i \in E_t$.  Thus $\gamma_i = 0$ for all $i=1,\dots,m$ as claimed.\qed
\end{proof}

\begin{remark}
  The condition that $A$ is infinite-dimensional cannot be removed from the reverse implication in Theorem~\ref{thm:Verma-reducible}.  Indeed, consider the case $A = \C$, so that $\cV = \Vir$.  If Theorem~\ref{thm:Verma-reducible} were true more generally, it would assert that $M(\varphi)$ is reducible if and only $\varphi(d_0)=0$.  However, this is not true.  For example, when $\varphi(c)=1$, $M(\varphi)$ is reducible if and only if $\varphi(d_0) = m^2/4$ for some $m \in \Z$ (see \cite[Proposition~8.3]{KR87}).
\end{remark}


\bibliographystyle{alpha}
\bibliography{savage-virasoro-biblist}

\begin{thebibliography}{HWZ03}

\bibitem[AM69]{AM69}
M.~F. Atiyah and I.~G. Macdonald.
\newblock {\em Introduction to commutative algebra}.
\newblock Addison-Wesley Publishing Co., Reading, Mass.-London-Don Mills, Ont.,
  1969.

\bibitem[Bou58]{Bou58}
N.~Bourbaki.
\newblock {\em \'{E}l\'ements de math\'ematique. 23. {P}remi\`ere partie: {L}es
  structures fondamentales de l'analyse. {L}ivre {II}: {A}lg\`ebre. {C}hapitre
  8: {M}odules et anneaux semi-simples}.
\newblock Actualit\'es Sci. Ind. no. 1261. Hermann, Paris, 1958.

\bibitem[BZ04]{BZ04}
Yuly Billig and Kaiming Zhao.
\newblock Weight modules over exp-polynomial {L}ie algebras.
\newblock {\em J. Pure Appl. Algebra}, 191(1-2):23--42, 2004.

\bibitem[CFK10]{CFK10}
Vyjayanthi Chari, Ghislain Fourier, and Tanusree Khandai.
\newblock A categorical approach to {W}eyl modules.
\newblock {\em Transform. Groups}, 15(3):517--549, 2010.

\bibitem[CP88]{CP88}
Vyjayanthi Chari and Andrew Pressley.
\newblock Unitary representations of the {V}irasoro algebra and a conjecture of
  {K}ac.
\newblock {\em Compositio Math.}, 67(3):315--342, 1988.

\bibitem[GLZ]{GLZ06}
Xiangqian Guo, Rencai Lu, and Kaiming Zhao.
\newblock Classification of irreducible {H}arish-{C}handra modules over
  generalized {V}irasoro algebras.
\newblock \emph{Proc. Edinb. Math. Soc.}\ (to appear), arXiv:math/0607614.

\bibitem[GLZ11]{GLZ08}
Xiangqian Guo, Rencai Lu, and Kaiming Zhao.
\newblock Simple {H}arish-{C}handra modules, intermediate series modules, and
  {V}erma modules over the loop-{V}irasoro algebra.
\newblock {\em Forum Math.}, 23(5):1029--1052, 2011.

\bibitem[HWZ03]{HWZ03}
Jun Hu, Xian-Dong Wang, and Kai-Ming Zhao.
\newblock Verma modules over generalized {V}irasoro algebras {${\rm Vir}[G]$}.
\newblock {\em J. Pure Appl. Algebra}, 177(1):61--69, 2003.

\bibitem[KR87]{KR87}
V.~G. Kac and A.~K. Raina.
\newblock {\em Bombay lectures on highest weight representations of
  infinite-dimensional {L}ie algebras}, volume~2 of {\em Advanced Series in
  Mathematical Physics}.
\newblock World Scientific Publishing Co. Inc., Teaneck, NJ, 1987.

\bibitem[Li04]{Li04}
H.~Li.
\newblock On certain categories of modules for affine {L}ie algebras.
\newblock {\em Math. Z.}, 248(3):635--664, 2004.

\bibitem[LZ06]{LZ06}
Rencai Lu and Kaiming Zhao.
\newblock Classification of irreducible weight modules over higher rank
  {V}irasoro algebras.
\newblock {\em Adv. Math.}, 206(2):630--656, 2006.

\bibitem[LZ10]{LZ10}
Rencai L{\"u} and Kaiming Zhao.
\newblock Classification of irreducible weight modules over the twisted
  {H}eisenberg-{V}irasoro algebra.
\newblock {\em Commun. Contemp. Math.}, 12(2):183--205, 2010.

\bibitem[Mat92]{Mat92}
Olivier Mathieu.
\newblock Classification of {H}arish-{C}handra modules over the {V}irasoro
  {L}ie algebra.
\newblock {\em Invent. Math.}, 107(2):225--234, 1992.

\bibitem[Maz99]{Maz99}
Vladimir Mazorchuk.
\newblock Verma modules over generalized {W}itt algebras.
\newblock {\em Compositio Math.}, 115(1):21--35, 1999.

\bibitem[Maz00]{Maz00}
Volodymyr Mazorchuk.
\newblock Classification of simple {H}arish-{C}handra modules over {${\bf
  Q}$}-{V}irasoro algebra.
\newblock {\em Math. Nachr.}, 209:171--177, 2000.

\bibitem[MP91a]{MP91a}
Christiane Martin and Alain Piard.
\newblock Indecomposable modules over the {V}irasoro {L}ie algebra and a
  conjecture of {V}. {K}ac.
\newblock {\em Comm. Math. Phys.}, 137(1):109--132, 1991.

\bibitem[MP91b]{MP91b}
Christiane Martin and Alain Piard.
\newblock Nonbounded indecomposable admissible modules over the {V}irasoro
  algebra.
\newblock {\em Lett. Math. Phys.}, 23(4):319--324, 1991.

\bibitem[MP92]{MP92}
Christiane Martin and Alain Piard.
\newblock Classification of the indecomposable bounded admissible modules over
  the {V}irasoro {L}ie algebra with weightspaces of dimension not exceeding
  two.
\newblock {\em Comm. Math. Phys.}, 150(3):465--493, 1992.

\bibitem[NS]{NS11}
Erhard Neher and Alistair Savage.
\newblock Extensions and block decompositions for finite-dimensional
  representations of equivariant map alegbras.
\newblock arXiv:1103.4367.

\bibitem[NSS]{NSS09}
Erhard Neher, Alistair Savage, and Prasad Senesi.
\newblock Irreducible finite-dimensional representations of equivariant map
  algebras.
\newblock \emph{Trans. Amer. Math. Soc.} (to appear), arXiv:0906.5189.

\bibitem[Rao04]{Rao04}
S.~Eswara Rao.
\newblock On representations of toroidal {L}ie algebras.
\newblock In {\em Functional analysis {VIII}}, volume~47 of {\em Various Publ.
  Ser. (Aarhus)}, pages 146--167. Aarhus Univ., Aarhus, 2004.

\bibitem[Su01]{Su01}
Yucai Su.
\newblock Simple modules over the high rank {V}irasoro algebras.
\newblock {\em Comm. Algebra}, 29(5):2067--2080, 2001.

\bibitem[Su03]{Su03}
Yucai Su.
\newblock Classification of {H}arish-{C}handra modules over the higher rank
  {V}irasoro algebras.
\newblock {\em Comm. Math. Phys.}, 240(3):539--551, 2003.

\end{thebibliography}

\end{document}